\documentclass{amsart}

\newif\ifdraft\draftfalse
\newif\ifcite\citefalse
\newif\ifblow\blowtrue

\makeatletter
\@namedef{subjclassname@2020}{\textup{2020} Mathematics Subject Classification}
\makeatother

\usepackage{amsfonts,amssymb, amsmath}
\usepackage{oldgerm}
\usepackage{url}
\usepackage{amscd}
\usepackage{mathrsfs} 
\usepackage[hyperindex]{hyperref}  
\usepackage[alphabetic]{amsrefs}
\ifdraft\ifcite\usepackage{showkeys}\else\usepackage[notcite,notref]{showkeys}\fi\fi
\usepackage{mathtools}

\usepackage{graphicx}
\usepackage{pinlabel}

\DeclarePairedDelimiter\floor{\lfloor}{\rfloor}

\newtheorem{proposition}[equation]{Proposition}
\newtheorem{theorem}[equation]{Theorem}

\theoremstyle{definition}

\theoremstyle{remark}

\newtheorem{remark}[equation]{Remark}

\newtheorem{rem}[equation]{Remark}
\newtheorem{question}[equation]{Question}

\newtheorem{example}[equation]{Example}

\numberwithin{equation}{section}

\newcommand\lp{\left(}
\newcommand\rp{\right)}

\def\bc{\begin{cases}}
\def\ec{\end{cases}}

\def\ol{\overline}

\def\t{\tilde}

\def\ba{{\mathbb A}}

\def\bc{{\mathbb C}}

\def\bh{{\mathbb H}}

\def\bn{{\mathbb N}}

\def\bq{{\mathbb Q}}

\def\bz{{\mathbb Z}}

\def\er{\eqref}

\def\bz{\mathbb Z}
\def\bq{\mathbb Q}

\def\bc{\mathbb C}

\def\bh{\mathbb H}

\def\bean{\begin{eqnarray}}
\def\eean{\end{eqnarray}}
\def\bea{\begin{eqnarray*}}
\def\eea{\end{eqnarray*}}
\def\beq{\begin{equation}}
\def\eeq{\end{equation}}
\def\bal{\begin{align*}}
\def\eal{\end{align*}}
\def\baln{\begin{align}}
\def\ealn{\end{align}}

\def\beg{\begin{gather*}}
\def\eng{\end{gather*}}
\def\bqu{\begin{question}}
\def\equ{\end{question}}

\def\implies{\Longrightarrow}

\def\ban{\begin{proof}[Answer]}
\def\ean{\end{proof}}

\def\on{\operatorname}

\def\bqu{\begin{question}}
\def\equ{\end{question}}

\def\0110{\begin{matrix} 0 & 1\\1&0\end{matrix}}

\def\t{\tilde}

\def\ban{\begin{proof}[Answer]}
\def\ean{\end{proof}}
\def\wt{\widetilde}

\def\ben{\begin{equation}}
\def\een{\end{equation}}

\def\j1{{(j+1)}}

\def\f32{{}_3F_2}

\newcommand{\zn}{{\mathbb Z}/N{\mathbb Z}}

\begin{document}

\title
{Cusps and boundaries of connected fundamental domains for $\Gamma_0(N)$}

\author{Zhaohu Nie}
\email{zhaohu.nie@usu.edu}
\address{Department of Mathematics and Statistics, Utah State University, Logan, UT 84322-3900, USA}


\begin{abstract} 
For $N>1$, we constructed a canonical connected fundamental domain for $\Gamma_0(N)$ in \cite{NP}, utilizing an interesting function $W: \bz/N\to \bn$. In this paper, we further study the function $W$, prove some identities, and use it to match the cusps, with widths, produced by our connected fundamental domain with the known cusp classes of $\Gamma_0(N)$. Furthermore, we list the boundary arcs and the gluing patterns of our connected fundamental domain, a key step in understanding the modular curve $X_0(N)$ by this approach. 
\end{abstract}

\maketitle

\section{Introduction}
The motivation for this paper came from the previous work \cite{NP}, in which we produced canonical connected fundamental domains for the congruence subgroups. Let $N>1$ and let us concentrate on the heart of that work, the $\Gamma_0(N)$ case. 

The connected fundamental domain produces some natural cusps with their own widths, specified by a function $W$ (see \er{def w}). The cusps produced this way are not inequivalent to each other. It was one motivation of this paper to classify these cusps by their equivalence classes and to reconcile the corresponding widths. 
The particular case of $N=30$ was worked out in \cite{NP}*{Example 3.3}, and here we aim to study the case for a general $N$ and to prove the corresponding identities. 

In the author's point of view, the advantage of a connected fundamental domain is the convenience it provides in understanding and getting a feel for the modular curve $X_0(N)$. As opposed to a disconnected fundamental domain consisting of ideal geodesic triangles in $\bh$, the common edges for a connected one are already identified, giving a clearer picture.  The remaining task is to specify how the boundary arcs are identified by elements of $\Gamma_0(N)$. Another goal of this paper is to write out the list of boundary arcs and the gluing patterns. This, a priori, seems a hopeless task for a general $N$, but the final result turns out to be rather neat. A key role is played by the geometry of the projective line $P^1(\bz/N)$ (see \er{P1}). 

After the motivation, let's describe these two results, about the cusps and boundaries of the connected fundamental domains, in more details.

A key tool in our construction \cite{NP} is a function 
\beq\label{def w}
W:\bz/N\to \bn;\quad j\mapsto W_j = \min\{m\in \bn\,|\, mj-1\in (\bz/N)^*\},
\eeq
where $\bn=\bz_{>0}$ and $(\bz/N)^*$ is the group of units. The related function
\beq\label{def m}
M: \bz/N\to \bz_{\geq 0};\quad j\mapsto M_j := W_j-1
\eeq
is actually first studied and more fundamental in our work \cite{NP}. 

Let 
$$S = \begin{pmatrix}
0 & -1\\
1 & 0
\end{pmatrix},
\quad
T = \begin{pmatrix}
1 & 1 \\
0 & 1
\end{pmatrix}
$$
be the generators of $\Gamma(1)=SL_2(\bz)$. 
Let  $A$ be a set in $\bz$ of consecutive residue class representatives for $\bz/N$, such as $\{0,1,\dots, N-1\}$, or 
\begin{equation}\label{sym chc}
\{-N_1,\dots,N_2\}, \quad N_1=\floor{\frac{N-1}{2}},\ N_2=\floor{\frac{N}{2}}, 
\end{equation}
where $\floor{\cdot}$ is the usual floor function. 
    For $x\in \bz $, we define $\wt{x}$ to be the unique integer such that $\wt{x}\equiv x \mod N$ and $-N_1\leq \wt{x}\leq N_2$.

The main result of \cite{NP} is that 
\beq\label{list}
\Theta = \{ST^i\,|\, i\in A\} \cup \{ST^jST^m\,|\, j\in A,\ \gcd(j, N)> 1,\ 0\leq m\leq M_j\}
\eeq 
is a set of right coset representatives for $\Gamma_0(N)\backslash \Gamma(1)$, which gives a connected fundamental domain for $\Gamma_0(N)$. 

\begin{figure}[ht]\label{fig 1}
\begin{center}
\includegraphics[scale=0.7]{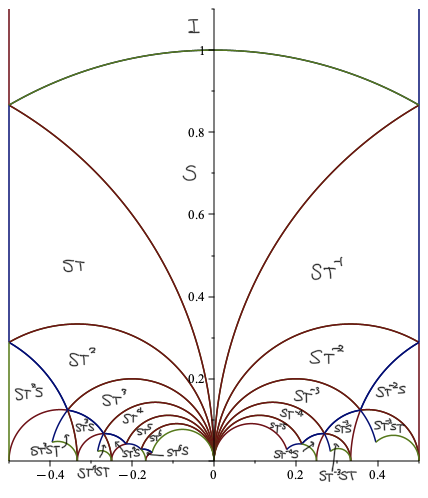}
\caption{Our Connected Fundamental Domain for $\Gamma_0(12)$}
\end{center}
\end{figure}

Figure \ref{fig 1} is the picture of the connected fundamental domain for $\Gamma_0(12)$, where the labels are at the corresponding images of the standard fundamental domain 
\begin{equation}\label{first D}
D = \Big\{z\in \bh\,\Big|\, |z|>1, |\on{Re}z|<\frac{1}{2}\Big\}.
\end{equation}
for $\Gamma(1)$ acting on the upper half plane $\bh$.

The author finds the function $W$ in \er{def w} interesting, and first establishes some identities for it. 

\begin{proposition}\label{wid}
We have 
\begin{align}
\sum_{j\in \bz/N} W_j &= \psi(N):=\prod_{p|N} \lp 1+\frac{1}{p}\rp\label{total}\\
\sum_{j\in (\bz/N)^*} W_j &= N,\label{1i}\\
\sum_{\gcd(j, N)>1} W_j &= \psi(N) - N.\label{jm}
\end{align}
\end{proposition}

 $\Gamma_0(N)$ acts on $P^1(\bq)=\bq\cup \infty$ by M\"obius transformations, and the quotient set is the set  $C_0(N)$ of cusp classes of $\Gamma_0(N)$.
The  cusp class of $s\in P^1(\bq)$ is denoted by  
$$
[s] =\Gamma_0(N)\cdot s\in C_0(\Gamma).
$$
An element $s=a/c\in P^1(\bq)$ with $a, c\in \bz$ and $\gcd(a, c)=1$ is called \emph{reduced}. By convention, $\infty=\pm 1/0$. 
 In this paper, we will only work with {reduced} elements.

In this paper, 
we use $d|N$ to denote that $d$ is a positive integer divisor of $N$. Then 
\beq\label{the ds}
\text{for }d|N,\quad d' := N/d,\quad d'':=\gcd(d, d'),\quad \tilde d := d'/d''.
\eeq


\begin{proposition}\label{cusps}
Let $N>1$. The number of cusp classes for $\Gamma_0(N)$ is 
$$\epsilon_\infty(\Gamma_0(N)) = \sum_{d|N} \phi(\gcd(d, N/d)) = \sum_{d|N} \phi(d''),$$
where  $\phi$ is the Euler totient function. 
Furthermore, with
$$S_0(N):=\bigcup_{d|N}(\bz/d'')^*,
$$
we have a bijection 
\beq\label{bij0}
\chi: C_0(N)\to S_0(N);\quad [a/c] \mapsto (d=\gcd(c, N); \pi_{d''}(a\cdot c/d)), 
\eeq
where $\pi_{d''}:\bz\to \bz/d''$ is the natural homomorphism. 
The width is 
$$
\on{wd}([a/c]) = \t d = d'/d''.
$$
\end{proposition}

This proposition can be verified using standard resources, such as \cite{DS}*{\S 3.8}, \cite{Shimura}*{\S 1.6}, \cite{Crem}*{\S 2.2} and \cite{S}*{\S 1.4}. We will not provide a detailed proof. 

The representatives $ST^i$ in \er{list} produce cusp 
$$
ST^i(\infty) = S(\infty) = 0.
$$
This in reduced form is $0/1$, and corresponds to $d=1$ in \er{bij0}. The width is $\t d = d'/d'' = N$, and this corresponds to that $i$ runs through the representatives for $\bz/N$ in \er{list}. 
This is the trivial part of our identification goal for cusps. 

On the other hand, the representatives $ST^jST^m$ in \er{list} produces cusps 
\beq\label{the cusp}
ST^jST^m(\infty) = \begin{pmatrix}
-1 & -m \\ j & mj-1 
\end{pmatrix}(\infty) = -\frac{1}{j}.
\eeq
Here $\gcd(j, N)>1$, and $0\leq m\leq M_j$. So the cusp $-1/j$ has a natural width $W_j$, in view of \er{def m}. 

If we let $j$ run through the residue class representatives $A=\{0,-1,\dots,-(N-1)\}$ of $\bz/N$, then 
$$
\{1/j \text{ with width }W_j\,|\, 0\leq j< N,\ \gcd(j, N)>1\}
$$ 
are natural cusps of  $\Gamma_0(N)$, produced by the work \cite{NP}. 

Now we achieve the goal of identifying these with the cusp classes in Proposition \ref{cusps}, and obtain the identity for widths. 

By \er{bij0}, we have
\beq\label{the class}
\chi([1/j])= (d; \pi_{d''}(j/d)),\quad \text{where }d=\gcd(j, N)>1.
\eeq
Going the other way, with $d>1$, $d|N$ and $b\in (\bz/d'')^*$, for $\chi([1/j])= (d; b)$, we need 
$$
d|j,\ j/d\in (\bz/d')^*,\ \pi^{*d'}_{d''} (j/d)=b,
$$
where $\pi^{*d'}_{d''}: (\bz/d')^*\to (\bz/d'')^*$ is the natural homomorphism. 
Therefore, we have  
\beq\label{kb}
j=dk,\quad k\in K_b := (\pi^{*d'}_{d''})^{-1}(b).
\eeq
where the notation $k$ is abused to also denote its integer lift for $\pi_{d'}: \bz\to \bz/d'$ between 1 and $d'$. 

We have the following result to relate the widths. 
\begin{theorem}\label{wdj} Let $N>1$, $d>1$ and $d|N$. Then the width of the cusp class for $\Gamma_0(N)$ represented by $(d; b\in (\bz/d'')^*)$ is the sum of the widths of all the $1/j$ such that $\chi([1/j])=(d;b)$, that is,  
\beq\label{hardest}
\t d=\frac{d'}{d''}=\sum_{k\in K_b} W_{dk}.
\eeq
\end{theorem}


Now to describe the boundary arcs of our fundamental domains, we consider the ideal geodesic triangle $\ol D$, the closure of $D$ in \er{first D} in the upper half plane $\bh$. We introduce the notation for its edges:
\begin{align*}
L &= \{z\in \bh\,|\, \on{Re}x=1/2,\ |z|\geq 1\},\\
R &= \{z\in \bh\,|\, \on{Re}x=-1/2,\ |z|\geq 1\},\\
B &= \{z\in \bh\,|\, |z|=1, |\on{Re}z|\leq {1}/{2}\},
\end{align*}
called the left, the right (from the viewpoint of the cusp $\infty$), and the base. 
Note that
\beq\label{idf0}
TR=L,\quad SB=B.
\eeq
Also, $L, R$ and their images under elements of $\Gamma(1)$ are all connected to cusps, while $B$ and its images are not.

In this part, we use the symmetric choice \er{sym chc} of residue classes, in agreement with \cite{NP}. 

\begin{proposition}\label{arcs} Let $N>1$. 
The boundary of the connected fundamental domain given by \er{list} has the following arcs: 
\begin{enumerate}
\item $ST^{N_2}L,\ ST^{-N_1}R$,
\item for $\gcd(i, N)=1$, $ST^i B$,
\item for $\gcd(j, N)>1$, 
\begin{enumerate}
\item $ST^jSR$,
\item $ST^jST^{M_j}L$,
\item $ST^jST^{m}B$, with $1\leq m\leq M_j$.  
\end{enumerate}
\end{enumerate}
\end{proposition}

We use the natural terminology that two boundary arcs $C_1$ and $C_2$ are equivalent, written as  
\beq\label{ceq}
C_1\sim C_2 \text{ if }\exists g\in \Gamma_0(N) \text{ such that }C_1 = gC_2.
\eeq 

\begin{theorem}\label{gluing}
We have the following gluing patterns for the boundary arcs of the connected fundamental domain for $\Gamma_0(N)$. 
\begin{enumerate}
\item $ST^{N_2}L \sim ST^{-N_1}R$.
\item For $\gcd(i, N)=1$, $ST^i B \sim ST^{\wt{-i^{-1}}} B$.
\item For $\gcd(j, N)>1$, 
\begin{enumerate}
\item $ST^jST^{M_j}L\sim ST^{\wt{(1-jW_j)^{-1} j}} SR$, and 
\item for $1\leq m\leq M_j$, $ST^jST^{m}B\sim ST^{j'}ST^{m'}B$ where 
\beq\label{orth}
jj' + (jm-1)(j'm'-1)\equiv 0\mod N.
\eeq
\end{enumerate}
\end{enumerate}
\end{theorem}

The boundary arcs of the fundamental domain for $\Gamma_0(12)$ in Figure \ref{fig 1}, besides the two vertical rays $\on{Re}z=\pm 1/2$ which are naturally identified by $T\in \Gamma_0(12)$, are listed below with equivalent pairs stacked, as an example of our Proposition \ref{arcs} and Theorem \ref{gluing}. 
\smallskip
\begin{center}
\begin{tabular}{|l|c|c|c|c|c|}
\hline
Arc & $ST^2SR$ &$ST^3STB$& $ST^3STL$& $ST^3SR$& $ST^4STB$  \\
\hline
Pair & $ST^{-2}STL$& $ST^{-2}STB$& $ST^{-3}SR$& $ST^{-3}STL$& $ST^{-3}STB$\\
\hline
\hline
Arc & $ST^4STL$ &$ST^4SR$& $ST^5B$& $ST^6SL$& $ST^6L$  \\
\hline
Pair & $ST^{-4}SR$& $ST^{-4}SL$& $ST^{-5}B$& $ST^6SR$& $ST^{-5}R$\\
\hline
\end{tabular}
\end{center}
\smallskip
The gluing pattern in this case turns out to be particularly simple, and this means that the corresponding modular curve $X_0(12)$ has genus 0 by standard topology knowledge. This is consistent with the general genus formula \cite{DS}*{Theorem 3.1.1}. 

The paper is organized as follows. In Section 2, we further study the function $W$, and 
prove Proposition \ref{wid}. 
In Section 4, we prove Theorem \ref{wdj} about the cusps. These two proofs are combinatorial and very concrete. 
In Section 5, we prove Theorem \ref{gluing} about the gluing patterns of the boundary arcs. 

\section{The function $W$}

First, we give a more concrete formula for the function $W$ \er{def w}. Let $N=p_1^{r_1}\dots p_t^{r_t}$ be its prime decomposition. 
Since $p_i\,|\,N$, 
let $\pi^N_{p_i}: \bz/N\to \bz/p_i$ be the natural projection.

\begin{proposition}\label{concrete}  Let $j\in \bz/N$, and  
$j_i= \pi^N_{p_i}(j)$ for $1\leq i\leq t$ as above. 
For an index $i$, if $j_i=0$, then disregard this index. For the others, let $\ell_i$ be the integer between 1 and $p_i-1$ representing $j_i^{-1}\in (\bz/p_i)^*$. 
Then 
$$
W_j = \min \bigg(\bn \big\backslash \bigcup_{p_i\nmid  j} \big(\ell_i + p_i\bz_{\geq 0}\big)\bigg).
$$
\end{proposition}


\begin{proof} An element $x\in \bz/N$ is a unit if $p_i\nmid x$ for $1\leq i\leq t$. 

Now if $p_i| j$, then any $m\in \bn$ would make $p_i\nmid mj-1$. 

If $\pi^N_{p_i}(j) \neq 0$, then we have its inverse $\ell_i$ as above. Then 
$$
p_i|(mj-1)\iff m\equiv \ell_i \mod p_i.
$$
So $m$ should avoid $\ell_i + p_i\bz_{\geq 0}$ for $mj-1$ to be not divisible by $p_i$. Putting these together, we get our result. 
\end{proof}

\begin{remark}
Therefore, we can design examples where $W_j$ for some $j$ is as big as we wish. For example, in $\bz/6$, $W_5 = 4$ as we need to avoid 
$$
\{1+2k, 2+3k\,|\, k\geq 0\}.
$$
\end{remark}

Now we go to the proof of the identities in Proposition \ref{wid}. Our proof depends on our previous work \cite{NP}, and let's first introduce the projective line
\beq\label{P1}
P^1(\bz/N)=\{(a,b)\,|\,a,b\in \bz/N\bz,\ \gcd(a,b,N)=1\}/\sim,
\eeq
where $(a,b)\sim (a',b')$ if there is $u\in (\bz/N\bz)^*$, such that $a=ua', b=ub'$. 
We write the class of $(a, b)$ by $(a:b)$.

\begin{proof}[Proof of Proposition \ref{wid}]
In \cite{NP}
, we proved that each class $(a:b)\in P^1(\bz/N)$ \er{P1} has a preferred element. That is, we have a bijection of sets
\beq\label{pref}
\on{pr}:P^1(\bz/N)\to \Phi := \{(j, jm-1)\,|\, j\in \bz/N,\ 0\leq m\leq M_j\},
\eeq
where $\on{pr}$ stands for preferred.



Since the cardinality of $P^1(\bz/N)$ is well known to be $\psi(N)$, we see that \er{total} holds by \er{def m} and \er{pref}. 

Note that the affine part 
\beq\label{aff}
\ba^1 :=\{(a:b)\,|\,a, b\in \zn,\ \gcd(a,N)=1\}=
\{(1:b)\,|\,b\in \bz/N\bz\}\subset P^1(\bz/N\bz)
\eeq
has cardinality $N$, and its image under the map in \er{pref} is 
$$
\on{pr}(\ba^1) = \{(j, mj-1)\,|\, j\in (\bz/N)^*, 0\leq m\leq M_j\}. 
$$
This establishes \er{1i}, since the cardinality of $\ba^1$ is clearly $N$. 

Then \er{jm} is the difference of these two. 

Now we present another direct proof of \er{1i}, whose idea we will continue to use in a harder situation. 

For definiteness, we choose the residue class representatives of $\bz/N$ to be $$\{0,1,\dots,N-1\}.$$ 
We list the integer representatives of the set $(\bz/N)^*$ of units in order as 
$$
u_1<u_2<\dots<u_{n},
$$
where $n=\phi(N)$. 

For $1\leq i\leq n$, we define 
$$
\Delta u_i = \begin{cases}
u_i - u_{i-1}, & 2\leq i\leq n,\\
u_1-(u_{n}-N)=N+u_1-u_n, & i=1.
\end{cases}
$$
We claim that 
$$
W_{u_i^{-1}} = \Delta u_i, \quad 1\leq i\leq n.
$$

The reason is that by our setup, 
$$
u_i-m\quad \text{is }
\begin{cases}
\text{not unit } & \text{if }1\leq m<\Delta u_i,\\
\text{a unit} & \text{if }m=\Delta u_i.
\end{cases}
$$
Multiplying by $u^{-1}_i$, we get 
$$
m u_i^{-1}-1 \quad \text{is }
\begin{cases}
\text{not unit } & \text{if }1\leq m<\Delta u_i,\\
\text{a unit} & \text{if }m=\Delta u_i.
\end{cases}
$$
This, by definition \er{def w}, means that 
\beq\label{neat}
W_{u^{-1}_i} = \Delta u_i.
\eeq

Therefore, 
$$
\sum_{j\in (\bz/N)^*} W_j= \sum_{i=1}^n W_{u^{-1}_i} = \sum_{i=1}^n \Delta u_i = N.
$$
\end{proof}

\begin{example}
For example, when $N=30$, we have 
\begin{center}
\begin{tabular}{|c|c|c|c|c|c|c|c|c|}
\hline
$u_i$ & 1 & 7 & 11 & 13 & 17 & 19 & 23 & 29 \\
\hline
$\Delta u_i$ & 2 & 6 & 4 & 2 & 4 & 2 & 4 & 6 \\
\hline
$u^{-1}_i$ & 1 & 13 & 11 & 7 & 23 & 19 & 17 & 29\\
\hline
$W_{u^{-1}_i}$ & 2 & 6 & 4 & 2 & 4 & 2 & 4 & 6 \\
\hline
\end{tabular}
\end{center}
Here the last row is computed by definition \er{def w}. For example, $W_{13}=6$ since $6\cdot 13 -1=77\in (\bz/30)^*$ is the first such instance.  We proved in \er{neat} that the second and the last rows are the same. 
\end{example}



\section{Cusps from the fundamental domains}


Now we go on to prove the identity relating the widths of the cusps of our fundamental domains with the widths of the cusp classes. 

\begin{proof}[Proof of Proposition \ref{wdj}]
First we show 
\beq\label{Kb}
K_b = \left\{\pi_{d'}(b + a_id) \in (\bz/d')^*
\,\Big|\,0\leq a_i<\t d=\frac{d'}{d''},\ 1\leq i\leq \frac{\phi(d')}{\phi(d'')}\right\},
\eeq
where $K_b$ is from \er{kb} with $b$ an integer between 1 and $d''-1$ representing the class in $(\bz/d'')^*$.
We first show that all the elements are distinct in the set on the RHS. For $ad\equiv a'd\mod d'$, we need $d'\,|\,(a-a')d$, so
$$
\frac{d'}{d''}\,\Big|\, (a-a')\frac{d}{d''}\implies \frac{d'}{d''}\,\Big|\, (a-a'),
$$
by $d''=\gcd(d, d')$. Therefore, for our range of $a$, the elements are distinct. Also the number of such $a_i$ is $\frac{\phi(d')}{\phi(d'')}$
by considering the homomorphism $\pi^{*d'}_{d''}: (\bz/d')^*\to (\bz/d'')^*$. 

Clearly the RHS of \er{Kb} is contained in $K_b$. Now since $d''=\gcd(d, d')$, we also see that all elements in $K_b$ have the form
$$
b+md'' = b + m(ud + vd') \equiv b+ mud\mod d',
$$
for some integers $m$, $u$ and $v$, hence belonging to the RHS. 

Then upon taking inverse, we have 
\beq\label{need inv}
K_{b^{-1}} = (K_b)^{-1}\subset (\bz/d')^*\quad \text{for }b^{-1}\in (\bz/d'')^*.
\eeq 

The following part is a more elaborate version of our direct proof of \er{1i}. We let $n=\frac{\phi(d')}{\phi(d'')}$, and order the $a_i$ in $K_{b}$ \er{Kb} as 
$$
a_1<a_2<\cdots<a_n.
$$
For $1\leq i\leq n$, we define 
$$
\Delta a_i = \begin{cases}
a_i - a_{i-1}, & 2\leq i\leq n,\\
a_1-(a_{n}-\t d)=\t d+a_1-a_n, & i=1.
\end{cases}
$$
We claim that 
$$
W_{(b+a_i d)^{-1}d} = \Delta a_i, \quad 1\leq i\leq n,
$$
where $\pi_{d'}(b+a_i d)\in (\bz/d')^*$, and $(b+a_i d)^{-1}$ is the integer between $1$ and $d'-1$ representing its inverse in $\bz/d'$. 

Note that for all $m\in \bn$, $m(b+a_i d)^{-1}d - 1\in (\bz/d)^*$, so we only need $m(b+a_id)^{-1} d-1\in (\bz/d')^*$ for it to be in $(\bz/N)^*$. Therefore, by \er{def w}, 
\beq\label{reduce}
W_{(b+a_i d)^{-1}d}^N = W_{(b+a_i d)^{-1}d}^{d'}.
\eeq
(Here we are using a superscript to identify the modulus for which $W$ \er{def w} is defined. The default one is $W=W^N$.)

By our setup, 
$$
b+a_i d -md\quad \text{is }
\begin{cases}
\text{not unit }\mod d' & \text{if }1\leq m<\Delta a_i,\\
\text{a unit}\mod d' & \text{if }m=\Delta a_i.
\end{cases}
$$
Multiplying by $(b+a_i d)^{-1}$ we get 
$$
m (b+a_i d)^{-1} d -1 \quad \text{is }
\begin{cases}
\text{not unit } \mod d'& \text{if }1\leq m<\Delta a_i,\\
\text{a unit} \mod d'& \text{if }m=\Delta a_i.
\end{cases}
$$
This, again by definition \er{def w}, means that 
\beq\label{again}
W_{(b+a_i d)^{-1}d}^{d'}= \Delta a_i.
\eeq


Therefore, from \er{need inv}, \er{reduce} and \er{again} we see that 
$$
\sum_{k\in K_{b^{-1}}} W_{kd}= \sum_{i=1}^{n} W_{(b+a_id)^{-1}d} = \sum_{i=1}^{n} \Delta a_i = \t d
$$
This is \er{hardest}, since $b^{-1}$ is arbitrary in $(\bz/d'')^*$. 
\end{proof}

\begin{example} Let $d=21, d'=90$. Then $N=d\cdot d' = 1890$, $d''=\gcd(d, d') = 3$, and $\t d = d'/d'' = 30$. Let us consider $b=1\in (\bz/d'')^*$. So the cusp class in $C_0(N)$ with invariant under $\chi$ \er{bij0} as  
$$(d; b)=(21; 1)\in S_0(N),$$
has width $\t d = 30$.  

The natural cusps \er{kb} from the fundamental domain are 
$$
\frac{1}{j} = \frac{1}{dk}, \quad k \in K_b = (\pi^{*90}_{3})^{-1} (1)=\{1, 7, 13, 19, 31, 37, 43, 49, 61, 67, 73, 79\} ,
$$
with $\phi(90)/\phi(3)=12$ element, where $\pi^{*d'}_{d''} : (\bz/90)^*\to (\bz/3)^*$ is the natural projection. These elements are in our third row below, and we are counting them as $(1+a_id)\mod d'$, with the $a_i$ increasing. 
\begin{center}
\begin{tabular}{|c|c|c|c|c|c|c|c|c|c|c|c|c|}
\hline
$a_i$ & 0 & 2 & 6 & 8 & 10 & 12 & 16 & 18 & 20 & 22 & 26 & 28  \\
\hline
$\Delta a_i$ & 2 & 2 & 4 & 2 & 2 & 2 & 4 & 2 & 2 & 2 & 4 & 2\\
\hline
$1+a_id \mod d'$ & 1 & 43 & 37 & 79 & 31 & 73 & 67 & 19 & 61 & 13 & 7 & 49 \\
\hline
$(1+a_id)^{-1}\mod d'$ & 1 & 67 & 73 & 49 & 61 & 37 & 43 & 19 & 31 & 7 & 13 & 79 \\
\hline
$(1+a_id)^{-1}d\mod d'$ & 21& 57& 3& 39& 21& 57& 3& 39& 21& 57& 3& 39\\
\hline
$W_{(1+a_id)^{-1}d}^{d'}$ & 2 & 2 & 4 & 2 & 2 & 2 & 4 & 2 & 2 & 2 & 4 & 2\\
\hline
\end{tabular}
\end{center}

So in this example, the last row of $W_{(1+a_id)^{-1}d}^{d'}$ can be computed from the second last row
by definition \er{def w}. 
Our result is that 
$$
W_{(1+a_id)^{-1}d}^{d'} = W_{(1+a_id)^{-1}d}^{N} = \Delta a_i,
$$
as in the second row, 
so they add to $\t d=30$. 
\end{example}


\section{Boundary identifications of the fundamental domains}

Identifying the boundary arcs in our connected fundamental domain is relatively straightforward. 

\begin{proof}[Proof of Proposition \ref{arcs}]
This follows from our construction using the set $\Theta$ \er{list} of coset representatives. First, we have the $ST^i D$ connected to one other. At the ends for $N_2$ and $-N_1$, we have case (1). 

For each $ST^i D$ with $(i, N)=1$, we have case (2) of $ST^i B$. 

For each $(j, N)>1$, we further attach $ST^j ST^m$, $0\leq m\leq M_j$. Again, at the two ends of $m=0$ and $m=M_j$, we have cases (3.a) and (3.b). For $1\leq m\leq M_j$, we have the images of the base as in case (3.c). Note that for $m=0$, $ST^jSB = ST^jB$ is identified in our connected fundamental domain already, and not a boundary arc. 

See Figure \ref{fig 1} for an example. 
\end{proof}

\begin{remark}
This proposition contains all the cases of $N>1$, including $N=2, 3, 4$. Also, when $j=0$, $M_j=0$, so we are in cases (3.a) and (3.b), and the two boundary arcs are 
$R$ and $L$. 
\end{remark}

Now, we introduce the following way of identifying the class of an element of $\Gamma(1)$ in the set $\Gamma_0(N)\backslash \Gamma(1)$ of right cosets . From \cite{Crem}*{\S 2.2}, we know that there is a bijection between $\Gamma_0(N)\backslash \Gamma(1)$ and $P^1(\bz/N)$, 
induced by the following map
\begin{equation}\label{the map}
R: \Gamma(1)\to P^1(\bz/N);\ \begin{pmatrix}a & b \\ c & d  \end{pmatrix}\mapsto (c:d),
\end{equation} 
where we choose the notation $R$ for row.

Therefore, for $\gamma_1, \gamma_2\in \Gamma(1)$, 
\begin{equation}\label{easy way}
[\gamma_1]=[\gamma_2]\in\Gamma_0(N)\backslash \Gamma(1)
\iff 
 R(\gamma_1)=R(\gamma_2)\in P^1(\bz/N).
\end{equation}

Then direct calculation gives, for $x, y\in \bz$, 
\begin{equation}\label{row rep}
\begin{split}
R(ST^x) &= (1:x),\\
R(ST^x S) &= (x: -1),\\
R( ST^x ST^y) &= (x: xy-1),\\
R(ST^x ST^y S) &= (xy-1: -x).
\end{split}
\end{equation}
Note that the right hand sides are guaranteed to be in $P^1(\bz/N)$. 

\begin{proposition}\label{intro st}
Let $\gamma_1,  \gamma_2\in \Gamma_1$ have different classes in $\Gamma_0(N)\backslash\Gamma(1)$. Using notation \er{ceq}, 
we have 
\begin{align}
\gamma_1 L\sim \gamma_2 R &\iff [\gamma_1T]=[\gamma_2]\in \Gamma_0(N)\backslash\Gamma(1),\text{ i.e. }\gamma_1T\gamma_2^{-1}\in \Gamma_0(N),\label{lr}\\
\gamma_1 B \sim \gamma_2 B&\iff [\gamma_1S]=[\gamma_2]\in \Gamma_0(N)\backslash\Gamma(1),\text{ i.e. }\gamma_1S\gamma_2^{-1}\in \Gamma_0(N).\label{bb}
\end{align}
\end{proposition}

\begin{proof} By $L=TR$ \er{idf0}, we see that $\gamma_1 L\sim \gamma_2 R$ iff $\exists g\in \Gamma_0(N)$ such that 
$$
\gamma_1 T R = g\gamma_2 R.
$$
Since only $\pm Id\in \Gamma(1)$ map $R$ to $R$, this is equivalent to  
$$
\gamma_2^{-1}g^{-1}\gamma_1 T = \pm Id,\text{ or } \gamma_1 T\gamma_2^{-1} = \pm g\in \Gamma_0(N).
$$
This proves \er{lr}, and \er{bb} is similar. Note that only $\pm Id$ and $\pm S$ preserve $B$, so from 
$\gamma_2^{-1} g^{-1} \gamma_1 B = B$, we see that 
$$
 \gamma_2^{-1} g^{-1} \gamma_1 = \pm Id\text{ or } \pm S.
 $$
 Therefore
 $$
 \gamma_1\gamma_2^{-1} = \pm g \text{ or } \gamma_1 S \gamma_2^{-1}  = \pm g.
 $$
 Because $\gamma_1$ and $\gamma_2$ have different classes, we are in the second case. 
\end{proof}

With these preparations done, let's go to the proof of our gluing patterns. 

\begin{proof}[Proof of Theorem \ref{gluing}]  We apply Proposition \ref{intro st}, and \er{easy way}, \er{row rep} throughout this proof. 

For (1), we need to check 
$$
[ST^{N_2}T] = [ST^{-N_1}],
$$
and this follows from 
$$
(N_2+1: 1)=(-N_1:1),\quad\text{by }N_1+N_2+1=N,
$$
from \er{sym chc}. 

For (2), we need to check 
$$
[ST^{i}S] = [ST^{\wt{-i^{-1}}}],
$$
and this follows from 
$$
(i:-1) = (1: -i^{-1}),\quad\text{by }\gcd(i, N)=1.
$$

For (3.a), we need to check 
$$
[ST^jST^{M_j}T] = [ST^{\wt{(1-jW_j)^{-1} j}}  S],
$$
and this follows from 
$$
(j: jW_j-1)  = ((1-jW_j)^{-1} j: -1),\quad\text{by }\gcd(jW_j-1, N)=1, 
$$
from \er{def m} and \er{def w}. 

Now note the elementary fact that 
\beq\label{cross mul}
(a:b)=(c:d)\in P^1(\bz/N)\iff ad = bc\in \bz/N.
\eeq
Only the $\Longleftarrow$ direction needs an explanation, and it can be proved by the existence of $u, v$ such that $au + bv = 1\in \bz/N$ by \er{P1}. We then multiply it by $c$ and $d$ separately, obtaining 
$$
c=(cu+dv)a,\quad d=(cu+dv)b,
$$
with $cu+dv\in (\bz/N)^*$ necessarily. 

Then for (3.b), we need to check 
$$
[ST^jST^{m}S] = [ST^{j'}ST^{m'}],
$$
and this follows from 
$$
(jm-1: -j) = (j': j'm'-1),
$$
which by \er{cross mul} is 
$$
(jm-1)(j'm'-1) = -jj' \in \bz/N,
$$
the assumption \er{orth}. 
\end{proof}



\begin{remark} In the case (3.b) of Theorem \ref{gluing}, we have 
\begin{equation}\label{1comp}
\gcd(j, N)>1,\quad \gcd(jm-1, N)>1,
\end{equation}
by $1\leq m\leq M_j$ and \er{def w}, \er{def m}. The corresponding $j'$ and $m'$ can be found by applying to $(jm-1: -j)$ our algorithm in \cite{NP}, which produces such representatives for all points in the hyperplane $H := P^1(\bz/N)\backslash \ba^1$ (see \er{aff}). Clearly  
$(jm-1: -j)\in H$ by \er{1comp}. 
\end{remark}


\begin{remark} In the case (3.a) of Theorem \ref{gluing}, we can see, by \er{the cusp}, that that cusps of $ST^jST^{M_j}$ and $ST^{\wt{(1-jW_j)^{-1} j}} S$ are
$$
-\frac{1}{j}\quad\text{and}\quad -\frac{1}{\wt{(1-jW_j)^{-1} j}}.
$$
By \er{the class}, their cusp classes are the same, since
$$
\gcd(1-jW_j, N)=1\quad\text{and}\quad \pi_{d''}((1-jW_j)^{-1}) = 1,
$$
where $d=\gcd(j, N), d''=\gcd(d, N/d)$ so $d''|j$. 

Therefore, the gluing pattern (3.a) brings together by $\Gamma_0(N)$ the cusps of the same class, as expected. This 
is compatible with the cusp width identity in Theorem \ref{wdj}. 
\end{remark}

\medskip 

\end{document}